\documentclass[12 pt,a4paper]{amsart} 
\usepackage{amsfonts,amssymb,amscd,amsmath,enumerate,verbatim,calc} 
\usepackage{float}
\usepackage{amsthm}
\newtheorem{theorem}{Theorem}[section]
\newtheorem{definition}{Definition}[section]

\newtheorem{lemma}[theorem]{Lemma}
\newtheorem{corollary}[theorem]{Corollary}
\newtheorem{proposition}[theorem]{Proposition}
\newtheorem{remark}[theorem]{Remark}
\usepackage{mathtools}

\newtheorem{example}[theorem]{Example}
\numberwithin{equation}{section}

\usepackage{amsfonts}
\newcommand{\field}[1]{\mathbb{#1}}          

\newcommand{\N}{\field{N}}
                   \newcommand{\Z}{\field{Z}}
\newcommand{\C}{\field{C}}

\newcommand{\im}{\rm Im}

\renewcommand{\ker}{\textnormal{ker}}
\renewcommand{\im}{\textnormal{Im}}

\usepackage{tikz-cd}

\begin{document}
	
	\title{Nilpotency and Capability in multiplicative Lie algebras}
	
	\author{Amit Kumar$^{1}$, Mani Shankar Pandey$^{2}$ AND Sumit Kumar Upadhyay$^{3}$\vspace{.4cm}\\
		{$^{1, 3}$Department of Applied Sciences,\\ Indian Institute of Information Technology Allahabad\\Prayagraj, U. P., India} \vspace{.3cm}\\ $^{2}$ Department of Sciences, \\Indian Institute of Information Technology \\Design and Manufacturing, Kurnool, Andhra Pradesh}
	
	\thanks{$^1$amitiiit007@gmail.com, $^2$manishankarpandey4@gmail.com,  $^3$upadhyaysumit365@gmail.com}
	

\begin{abstract}
	This paper aims to introduce  the concept of nilpotency and capability in multiplicative Lie algebras. Also, we see the existence of covers of a multiplicative Lie algebra and thoroughly examine their relationships with capable and  perfect multiplicative Lie algebras. 
	\end{abstract}

\maketitle

\section{Introduction} 
A multiplicative Lie algebra is a new algebraic structure introduced by Ellis \cite{GJ} with the primary objective of characterizing universal commutator identities. Building upon the concept of multiplicative Lie algebra, Point and Wantiez \cite{FP} further developed the notions of solvability and nilpotency. Exploring the proximity of multiplicative Lie algebras to improper multiplicative Lie algebras, Pandey, Lal, and Upadhyay \cite{MRS} introduced the concepts of Lie commutator and multiplicative Lie center, leading to a new understanding of solvable and nilpotent multiplicative Lie algebras. In this article, we present a novel idea of nilpotency in multiplicative Lie algebras which implies the existing concepts of nilpotency in \cite{FP} and \cite{MRS}. The theory of multiplicative Lie algebras can be further developed in the direction of existing theory for groups as well as Lie algebra with the aid of this new notion of nilpotency.



The concept of capability for a group, first introduced by R. Baer \cite{Baer}, involves the systematic investigation of conditions under which a group $G$ can serve as the group of inner automorphisms of another group $E$ $(G \cong E/Z(E))$. A group that satisfies this criterion is commonly referred to as a capable group. It is evident that this theory can be extended to other algebraic structures, such as Lie algebras or multiplicative Lie algebras, where an analogous notion of capability can be developed. Notably, significant progress has already been made in the study of capable Lie algebras \cite{PB93, SAM}.  In this article, we concentrate on the study of capable multiplicative Lie algebras, which we refer to as Lie capable. In this direction, we define a smallest ideal $\mathcal{Z}^*(G)$ of a multiplicative Lie algebra $G$ such that the quotient structure $\frac{G}{\mathcal{Z}^*(G)}$ is  Lie capable . Through these investigations, we contribute to the understanding of capability within multiplicative Lie algebras and shed light on the significance of the ideal $\mathcal{Z}^*(G)$ in characterizing the capability of such algebras. Our investigation delves into the relationship between the Schur multiplier and the capability of a multiplicative Lie algebra. In addition, we explore the concept of cover of a multiplicative Lie algebra, as defined in \cite{MS}. Notably, we establish the existence of a cover for every multiplicative Lie algebra and provide an example illustrating that the cover of a multiplicative Lie algebra is not necessarily unique.  We also discussed relations between free presentations and covers of a  multiplicative Lie algebra. Lastly, we explore the relationship between covers and the Schur multipliers of finite perfect multiplicative Lie algebras. 
\section{Preliminaries} 
In this section, we provide a review of key definitions and results pertaining to multiplicative Lie algebras, which will be  used in this paper.
\begin{definition}\cite{GJ}
	A multiplicative Lie algebra is a triple $ (G,\cdot,\star), $ where $G$ is a set, $ \cdot $ and $ \star $ are two binary operations on $G$ such that $ (G,\cdot) $ is a group (need not be abelian) and for all  $x, y, z  \in  G$, the following identities hold: 
	\begin{enumerate}
		\item $ x\star x=1 $  
		\item $ x\star(y \cdot z)=(x\star y)\cdot{^y(x\star z)} $ 
		\item $ (x \cdot y)\star z= {^x(y\star z)} \cdot (x\star z) $ 
		\item $ ((x\star y)\star {^yz})((y\star z)\star{^zx})((z\star x)\star{^xy})=1 $ 
		\item $ ^z(x\star y)=(^zx\star {^zy})$ 
	\end{enumerate}
	where $^xy$ denotes $xyx^{-1}.$ We say that $ \star $ is a multiplicative Lie algebra structure on the group $G$.
\end{definition}
\begin{definition}
Let $(G,\cdot,\star)$ be a multiplicative Lie algebra. Then
\begin{enumerate}
\item A subgroup $H$ of $G$ is said to be subalgebra of G if $x\star y \in H$ for all $x,y \in H$.

\item A subalgebra $H$ of $G$ is said to be an ideal of $G$ if it is a normal subgroup of $G$ and $x\star y \in H$ for all $x \in G$ and $y \in H$. The ideal generated by $\langle a\star b ~\mid ~ a, b \in G\rangle $ is denoted by $G\star G$.

\item Let $(G',\circ , \star')$ be another multiplicative Lie algebra. A group homomorphism $\psi: G \longrightarrow G'$ is called a multiplicative Lie algebra homomorphism if $\psi(x\star y) =\psi(x) \star ' \psi(y)$ for all $x, y \in G$.

\item The ideal	$ LZ(G) = \{x \in G | x \star  y = 1 $ for all $y\in G\}$  is called the Lie center of $G.$ 
\item The ideal	$ MZ(G) = \{x \in G | x \star  y = [x,y] $ for all $y\in G\}$  is termed as Multiplicative Lie center of $G.$
\end{enumerate}
\end{definition}
Ellis \cite{GJ} shows that the multiplicative Lie product ``$\star$" satisfies some identities. In 2019, 
Lal and Upadhyay \cite{RLS}, rephrase the same identities. 
\begin{proposition}\cite{GJ,RLS}
	Let $(G,\cdot,\star)$ be a  multiplicative Lie algebra. Then the following  identities hold:
	\begin{enumerate}
		\item $ (1\star x) = 1 = (x\star 1) $ for all $x\in G.$
		
		\item $ (x\star y)(y\star x) = 1 $  for all $x,y \in G.$
		
		\item $ ^{(x\star y)}(u\star v) = {^{[x,y]}}(u\star v) $ for all $x,y,u,v \in G.$
		
		\item $ [(x\star y),z] = ([x,y]\star z) $ for all  $x,y,z \in G.$
		
		\item $ x^{-1}\star y = {^{x^{-1}}(x\star y)^{-1}}$ and $ x\star y^{-1} = {^{y^{-1}}(x\star y)^{-1}} $ for all $x,y \in G.$
		\end{enumerate}
\end{proposition}
The following definition of a perfect multiplicative Lie algebra has been taken from \cite{RLS}.
\begin{definition}\cite{RLS}
	A multiplicative Lie algebra $G$ is said to be perfect  if
	$${(G\star G)[G,G]}=G.$$
\end{definition}
\begin{definition}\cite{RLS}
A short exact sequence of multiplicative Lie algebras
 $$1\longrightarrow H \longrightarrow G \longrightarrow K\longrightarrow 1$$
is said to be a central extension if $H\subseteq Z(G)\cap LZ(G).$
\end{definition}
\begin{proposition}\cite{RLS}\label{P1} 
	Let  $G$ be a perfect multiplicative Lie algebra, and let $$1\longrightarrow H \longrightarrow K \longrightarrow G\longrightarrow 1$$ be a central extension by $G$. Then the subalgebra ${(K\star K)[K,K]}$ is perfect, and $$1\longrightarrow H\cap ({(K\star K)[K,K]}) \longrightarrow {(K\star K)[K,K]} \longrightarrow G\longrightarrow 1$$ is also a central extension by $G$.
\end{proposition}

\begin{proposition}\cite{RLS}\label{P2} 
	Every perfect multiplicative Lie algebra $G$ admits (of course, a unique) universal central extension.
\end{proposition}

\section{Nilpotent multiplicative Lie algebras }

In this section, we introduce a new idea of nilpotency in multiplicative Lie algebras. Also, we relate it to the Schur multiplier of a multiplicative Lie algebra.  For that, we need the following notions.

	\begin{proposition}
		Let $(G,\cdot,\star)$ be a multiplicative Lie algebra and $H$ be an ideal of $G$. Then $(G\star H)[G,H]$ is an ideal of $G.$ 
	\end{proposition}

\begin{proof}
	
	Let $ g \in G $ and $ (a\star b)[c,d] \in(G\star H)[G,H]. $ Then 
\begin{align*}
g((a\star b)[c,d])g^{-1} &= (g(a\star b)g^{-1})(g[c,d]g^{-1}) \\ &= (^ga\star ^gb)[^gc,^gd] \in(G\star H)[G,H].
\end{align*}
This implies $(G\star H)[G,H]$ is a normal subgroup of $G.$ Further
\begin{align*}
\begin{split}
g\star((a\star b)[c,d]) &= (g\star(a\star b))^{(a\star b)}(g\star [c,d]) \\ &= (g\star(a\star b))^{(a\star b)}[g,(c\star d)] \\ &= (g\star(a\star b))[^{(a\star b)}g,^{(a\star b)}(c\star d)] \in (G\star H)[G,H].
\end{split}
\end{align*}
This implies $(G\star H)[G,H]$ is an ideal of $G.$

\end{proof}

\begin{remark}
For the sake of brevity and clarity, we adopt the notation $^M[G,H]$ to represent the ideal $(G\star H)[G,H]$, going forward.
\end{remark}

Define $M_n(G)$ inductively as follows: $M_0(G) = G.$ Assuming that $M_n(G)$ has already been defined, define $M_{n+1}(G) = ^M[G,M_n(G)].$ This gives us a descending chain:
\begin{align*}
	G = M_0(G)\unrhd M_1(G)\unrhd M_2(G)\unrhd \cdots \unrhd M_n(G) \unrhd \cdots
\end{align*}
of ideals of $G$ which is termed as lower central series of $G.$

\begin{remark}
	We denote $ LZ(G)\cap Z(G) $  by $ \mathcal{Z}(G). $ 
\end{remark}

Next define ideals $\mathcal{Z}_n(G)$ of $G$ inductively as $\mathcal{Z}_0(G) = \{1\}$, $\mathcal{Z}_1(G) = \mathcal{Z}(G).$ 
Let $\phi: G \rightarrow \frac{G}{\mathcal{Z}(G)}$ be a natural surjective multiplicative Lie algebra homomorphism.
Define $\mathcal{Z}_2(G) = \phi^{-1}(\mathcal{Z}(\frac{G}{\mathcal{Z}(G)})).$ This implies $\frac{\mathcal{Z}_2(G)}{\mathcal{Z}_1(G)} = \mathcal{Z}(\frac{G}{\mathcal{Z}(G)}).$ Similarly,
define  $\mathcal{Z}_{n+1}(G) = \phi^{-1}(\mathcal{Z}(\frac{G}{\mathcal{Z}_n(G)})).$ This implies $\frac{\mathcal{Z}_{n+1}(G)}{\mathcal{Z}_n(G)} = \mathcal{Z}(\frac{G}{\mathcal{Z}_n(G)}).$  This gives us an ascending chain:
\begin{align*}
\{1\} =\mathcal{Z}_0(G)\unlhd \mathcal{Z}_1(G)\unlhd \mathcal{Z}_2(G)\unlhd \cdots \unlhd \mathcal{Z}_n(G) \unlhd \cdots
\end{align*}
of ideals of $G$ which is termed as the upper central series of $G.$

\begin{theorem}
	The lower central series of $G$ terminates to $\{1\}$ at the $nth$ step if and only if the upper central series of $G$ terminates to $G$ at the $nth$ step.
\end{theorem}

\begin{proof}
	Suppose $M_n(G) = \{1\}.$ To prove $ \mathcal{Z}_n(G) = G $, we show inductively that $M_{n-i}(G) \subseteq \mathcal{Z}_i(G)$ for all $1\leq i \leq n.$ For $i=0$ the statement holds. Suppose it is true for $i$, that is $M_{n-i}(G) \subseteq \mathcal{Z}_i(G).$ Let $a\in G$, $b\in M_{n-i-1}(G). $ Then $(a\star b) \in M_{n-i}(G)$, $[a,b] \in M_{n-i}(G).$ Thus we have $a\mathcal{Z}_i(G)\star b\mathcal{Z}_i(G) = (a\star b)\mathcal{Z}_i(G) = \mathcal{Z}_i(G). $ This implies $b\mathcal{Z}_i(G) \in LZ(\frac{G}{\mathcal{Z}_i(G)}).$ Also $[a\mathcal{Z}_i(G),b\mathcal{Z}_i(G)] = [a,b]\mathcal{Z}_i(G) = \mathcal{Z}_i(G).$ This implies  $b\mathcal{Z}_i(G) \in Z(\frac{G}{\mathcal{Z}_i(G)}).$ Hence $b\mathcal{Z}_i(G) \in LZ(\frac{G}{\mathcal{Z}_i(G)}) \cap Z(\frac{G}{\mathcal{Z}_i(G)}) = \mathcal{Z}(\frac{G}{\mathcal{Z}_i(G)}) = \frac{\mathcal{Z}_{i+1}(G)}{\mathcal{Z}_i(G)}. $ This implies $b \in \mathcal{Z}_{i+1}(G).$ We proved that $M_{n-i}(G) \subseteq \mathcal{Z}_i(G)$ for all $1\leq i \leq n.$ So $G = M_{n-n}(G)\subseteq \mathcal{Z}_n(G) \Rightarrow \mathcal{Z}_n(G) = G. $ 
	
	Conversely, suppose that $ \mathcal{Z}_n(G) = G. $  Again, we  show inductively that  $M_i(G) \subseteq \mathcal{Z}_{n-i}(G)$ for all $1\leq i \leq n.$ For $i=0$ the statement is true. Suppose it is true for $i$, that is $M_i(G) \subseteq \mathcal{Z}_{n-i}(G).$ Let $a,c \in G$ and $b,d \in M_i(G).$ Then $(a\star b)[c,d] \in M_{i+1}(G).$ Since  $b,d \in M_i(G)\subseteq \mathcal{Z}_{n-i}(G) $, we have $b\mathcal{Z}_{n-i-1}(G), d\mathcal{Z}_{n-i-1}(G) \in \frac{\mathcal{Z}_{n-i}(G)}{\mathcal{Z}_{n-i-1}(G)} = \mathcal{Z}(\frac{G}{\mathcal{Z}_{n-i-1}(G)}) =  LZ(\frac{G}{\mathcal{Z}_{n-i-1}(G)}) \cap Z(\frac{G}{\mathcal{Z}_{n-i-1}(G)}).$ This implies  $a\mathcal{Z}_{n-i-1}(G)\star  b\mathcal{Z}_{n-i-1}(G) = \mathcal{Z}_{n-i-1}(G)$ and so $(a\star b)\in \mathcal{Z}_{n-i-1}(G). $
	Also $[ c\mathcal{Z}_{n-i-1}(G), d\mathcal{Z}_{n-i-1}(G) ]= \mathcal{Z}_{n-i-1}(G)$ and so $[c,d] \in \mathcal{Z}_{n-i-1}(G). $ Hence $(a\star b)[c,d] \in\mathcal{Z}_{n-i-1}(G). $  We proved that $M_i(G) \subseteq \mathcal{Z}_{n-i}(G)$ for all $1\leq i \leq n.$ So $ M_n(G)\subseteq \mathcal{Z}_0(G) = \{1\} \Rightarrow  M_n(G) = \{1\}. $ 
	
\end{proof}

Thus, we have the following definition of nilpotency in multiplicative Lie algebras.

\begin{definition}
	A  multiplicative Lie algebra $(G,\cdot,\star)$ is said to be $M\mathcal{Z}$-nilpotent if $M_n(G) = \{1\} $ or equivalently $ \mathcal{Z}_n(G) = G $ for some $n \in \N.$ It is said to be $M\mathcal{Z}$-nilpotent of class $n$ if $M_n(G) = \{1\} $ but $M_{n-1}(G) \neq \{1\} $ or equivalently $ \mathcal{Z}_n(G) = G $ but $ \mathcal{Z}_{n-1}(G) \neq G. $ 
\end{definition}

\begin{remark}
	\begin{enumerate}
		
		\item  If we consider a Lie simple group $G$, it can be observed that the concept of an $M\mathcal{Z}$-nilpotent multiplicative Lie algebra is equivalent to that of a nilpotent group.
		
		\item Suppose we have a multiplicative Lie algebra $(G, \cdot, \star)$, where the underlying group structure $(G, \cdot)$ is abelian. In this case, it can be observed that the concept of an $M\mathcal{Z}$-nilpotent multiplicative Lie algebra aligns with the nilpotency criteria established by Point and Wantiez \cite{FP}.
		\item Consider an improper multiplicative Lie algebra $(G, \cdot, \star)$, where the underlying group structure $(G, \cdot)$ is non-abelian and non-nilpotent. In this case, it can be concluded that $G$ is Lie nilpotent, indicating the successive iterated commutators of elements in $G$ eventually become trivial. However, it is important to note that $G$ is neither $M\mathcal{Z}$-nilpotent nor nilpotent as a multiplicative Lie algebra.
	\end{enumerate}
\end{remark}
The proofs of the traditional results discussed in this section bear resemblance to the corresponding proofs for groups.
\begin{proposition}
	Let $(G,\cdot,\star)$ be $M\mathcal{Z}$-nilpotent multiplicative Lie algebra. Then
	\begin{enumerate}
	\item $(G,\cdot,\star)$  is  Lie nilpotent  multiplicative Lie algebra.
	\item $(G,\cdot,\star)$ is nilpotent  multiplicative Lie algebra.
	\item $(G,\cdot)$ is nilpotent group.
	\end{enumerate}
  But the converse need not be true.
\end{proposition}

\begin{example}
	\begin{enumerate}
		\item Consider the group $A_4 = <a,b:a^3 = b^2 = 1, aba = ba^2b>$ with improper multiplicative Lie algebra. Then it is Lie nilpotent  multiplicative Lie algebra. But $A_4$ is not a nilpotent group therefore it is not $M\mathcal{Z}$-nilpotent multiplicative Lie algebra.
		
		\item Consider a perfect group  $G$ $(i.e. [G,G] = G)$ with trivial  multiplicative Lie algebra structure. Since $G\star G = \{1\}. $ Therefore $G$ is nilpotent multiplicative Lie algebra but not $M\mathcal{Z}$-nilpotent multiplicative Lie algebra.

	    \item Let $V_4 = <a,b:a^2 = b^2 = 1, ab = ba>$ be the Klein four group with  multiplicative Lie algebra structure $a\star b = b.$ Then $V_4$ is nilpotent group but not  $M\mathcal{Z}$-nilpotent multiplicative Lie algebra.
    	
    	\item  Consider the dihedral group $D_4 = \langle a, b \mid a^2 = b^4 = 1, ab = b^{-1}a \rangle$ with the multiplicative Lie algebra structure defined by $a\star b = b$.
By calculating the ideals, we find that $\mathcal{Z}_0(G) = {1}$ and $\mathcal{Z}_1(G) = LZ(G)\cap Z(G) = {1}$. Consequently, $\mathcal{Z}_n(G) = {1}$ for all $n\in \mathbb{N}$.
Hence, it can be concluded that $(D_4,\cdot,\star)$ is not an $M\mathcal{Z}$-nilpotent multiplicative Lie algebra.

Similarly, consider the group $A_4 = \langle a, b \mid a^3 = b^2 = 1, aba = ba^2b \rangle$ with the multiplicative Lie algebra structure defined by $a\star b = b$. In this case as well, the ideals $\mathcal{Z}_n(G)$ are trivial, i.e., $\mathcal{Z}_n(G) = {1}$ for all $n\in \mathbb{N}$.
Thus, it can be concluded that $(A_4,\cdot,\star)$ is not an $M\mathcal{Z}$-nilpotent multiplicative Lie algebra.
    	   \end{enumerate}
        \end{example}

Now we will state following traditional results for $M\mathcal{Z}$-nilpotent multiplicative Lie algebras.
\begin{proposition}
	\begin{enumerate}
		\item Subalgebra of a  $M\mathcal{Z}$-nilpotent multiplicative Lie algebra is
	 $M\mathcal{Z}$-nilpotent.
	 
	 \item Homomorphic image of a $M\mathcal{Z}$-nilpotent multiplicative Lie algebra is $M\mathcal{Z}$-nilpotent. In particular, the quotient of a $M\mathcal{Z}$-nilpotent multiplicative Lie algebra is $M\mathcal{Z}$-nilpotent.
	 
	 \item Let $H$ be an ideal of $G$ contained in $\mathcal{Z}(G)$ such that $\frac{G}{H}$
	  is $M\mathcal{Z}$-nilpotent. Then $G$ is also $M\mathcal{Z}$-nilpotent.
	 
	 \item A multiplicative Lie algebra $G$ is $M\mathcal{Z}$-nilpotent if and only if $\frac{G}{\mathcal{Z}(G)}$ is $M\mathcal{Z}$-nilpotent.

	 \item Let $G$ be a $M\mathcal{Z}$-nilpotent multiplicative Lie algebra. Then $\mathcal{Z}(G) \neq \{1\}.$
	 
	 \item The direct product of a finite number of  $M\mathcal{Z}$-nilpotent multiplicative Lie algebras is $M\mathcal{Z}$-nilpotent.
	 
	 \item Let $H$ and $K$ be ideals of a multiplicative Lie algebra $G$ such that $\frac{G}{H}$
	  and $\frac{G}{K}$ are $M\mathcal{Z}$-nilpotent of classes $n$ and $m$, respectively. Then $\frac{G}{H\cap K}$ is also  $M\mathcal{Z}$-nilpotent of class at most $max(n, m).$
	\end{enumerate}
\end{proposition}

\begin{theorem}
Let $H$ and $K$ be  two  $M\mathcal{Z}$-nilpotent ideals of $G$ of classes $n$ and $m$, respectively such that $\mathcal{Z}(H)$ and $\mathcal{Z}(K)$ are also ideals of $G.$  Then $HK$ is a $M\mathcal{Z}$-nilpotent ideal of $G$ of class at most $n + m.$
\end{theorem}

\begin{remark}
	Let $G$ be a  multiplicative Lie algebra with $ (G\star G) \subseteq Z(G) $.Then 
	
	\begin{enumerate}
		\item $(a^{-1}\star b) = (a\star b^{-1}) = (a\star b)^{-1}$ for all $a,b \in G.$
		
		\item $(a^n\star b) = (a\star b^n) = (a\star b)^n$ for all $a,b \in G, n \in \Z.$
	\end{enumerate}

In particular, if $G$ is $M\mathcal{Z}$-nilpotent multiplicative Lie algebra of class $2$, then 

	\begin{enumerate}
	\item $(a^{-1}\star b) = (a\star b^{-1}) = (a\star b)^{-1}$ for all $a,b \in G.$
	
	\item $(a^n\star b) = (a\star b^n) = (a\star b)^n$ for all $a,b \in G, n \in \Z.$
   \end{enumerate}
\end{remark}

Define $M^n(G)$ inductively as follows: $M^0(G) = G.$ Assuming that $M^n(G)$ has already been defined, define $M^{n+1}(G) = {^M[M^n(G),M^n(G)]}.$ This gives us a descending chain:
\begin{align*}
G = M^0(G)\unrhd M^1(G)\unrhd M^2(G)\unrhd \cdots \unrhd M^n(G) \unrhd \cdots
\end{align*}
of ideals of $G$ which is termed as the derived series of $G.$

\begin{definition}
	 A multiplicative Lie algebra $G$ is said to be  $M$-solvable if the
	derived series of $G$ terminates to $\{1\}$ after finitely many terms. The smallest $n$ such
	that $M^n(G) = \{1\}$ is called the derived length and the series
	$$G = M^0(G)\unrhd M^1(G)\unrhd M^2(G)\unrhd \cdots \unrhd M^n(G) = \{1\}$$ is called the $M$-solvable series of $G.$
\end{definition}

\begin{remark}
	Any $M\mathcal{Z}$-nilpotent multiplicative Lie algebra is $M$-solvable. But the converse need not be true.
\end{remark}

\begin{example}
Consider the symmetric group $S_3$ with improper multiplicative Lie algebra. Then it is  $M$-solvable but not  $M\mathcal{Z}$-nilpotent.  
\end{example}

\section{Capable multiplicative Lie algebras}

In this section, we introduce the concept of capability in multiplicative Lie algebras. After demonstrating that covers exist in each multiplicative Lie algebra, we tie capability to covers and Schur multipliers. Finally, we give connections among perfect multiplicative Lie algebras, covers and Schur multipliers.  

\begin{definition}
	\begin{enumerate}
		\item A  multiplicative Lie algebra $(G,\cdot,\star)$ is called  Lie capable if there exists a  multiplicative Lie algebra $(E,\cdot,\star)$ such that $G \cong \frac{E}{\mathcal{Z}(E)}.$
		
		\item A  multiplicative Lie algebra $G$ is called  multiplicative capable if there exists a  multiplicative Lie algebra $E$ such that  $G \cong \frac{E}{\mathcal{Z}(E)}$ with $ \mathcal{Z}(E) = {MZ}(E).$
		
		\item A  multiplicative Lie algebra  $G$ is called  Lie unicentral if the center of every central extension of $G$ maps onto $\mathcal{Z}(G).$
	\end{enumerate}
	
\end{definition}

\begin{remark}
	Every group $G$ that is capable is also Lie capable and multiplicative capable under the trivial multiplicative Lie algebra structure.
\end{remark}

\begin{example}
\begin{enumerate}
\item Since there is no multiplicative Lie algebra structure on an abelian group $G$ such that $\frac{G}{LZ(G)}$ is cyclic, non-trivial cyclic group is not Lie capable.

\item  Consider the Klein four group $V_4 = \langle a, b \mid a^2 = b^2 = 1, ab = ba \rangle$ with the multiplicative Lie algebra structure defined by $a \star b = b$. In this case, we observe that both the Lie center and the multiplicative Lie center of $V_4$ are trivial, denoted by $\mathcal{Z}(V_4) = {MZ}(V_4) = {1}$. As a result, we conclude that $V_4$ is both Lie capable and multiplicative capable.
\end{enumerate}

\end{example}

Here we give a necessary condition for a multiplicative Lie algebra to be multiplicative capable:
\begin{proposition}
If $G$ is multiplicative capable multiplicative Lie algebra  and the factor group $\frac{G}{^M[G,G]}$  is of finite exponent, then $\mathcal{Z}(G)$ is bounded and exponent of $\mathcal{Z}(G)$ divides that of $\frac{G}{^M[G,G]}.$ 
\end{proposition}
\begin{proof}
Suppose $G \cong \frac{E}{\mathcal{Z}(E)}$ for some  multiplicative Lie algebra $E$ with $ \mathcal{Z}(E) = {MZ}(E).$  Thus we have a central extension
$$\xi: 1 \longrightarrow \mathcal{Z}(E) \longrightarrow E \overset{\phi}\longrightarrow G\longrightarrow 1$$ of $G.$ Let $t$ be a section of $\xi$ satisfying $(\phi \circ t)(x)=x$ for all $x\in G$, and $t(1)=1.$ If $s$ is another section of $\xi$, then $s(x)=z_x t(x)$ for some $z_x \in \mathcal{Z}(E). $ Thus $[t(x),t(y)](t(x')\star t(y'))=[s(x),s(y)](s(x')\star s(y')).$ The Lie commutators are independent of the chosen section. Thus we have a coordinatewise group homomorphism $\gamma:\mathcal{Z}(G)\times G \to\mathcal{Z}(E) $ given by $\gamma(x,y)=[t(x),t(y)](t(y)\star t(x)).$ Fix $x \in \mathcal{Z}(G)$ then we have a group homomorphism $\gamma_x: G \to \mathcal{Z}(E) $ given by $\gamma_x(y)=[t(x),t(y)](t(y)\star t(x)).$ The right kernel of $\gamma$ is $\ker_R(\gamma)=\cap_{x\in \mathcal{Z}(G)}\ker(\gamma_x).$ Since $[G,G]\subseteq \ker_R(\gamma) $ and $(G\star G)\subseteq \ker_R(\gamma) $, we have  $^M[G,G]\subseteq \ker_R(\gamma). $
Since exponent of $\frac{G}{^M[G,G]}$ is finite, say $n$, we have $\gamma(x,\bar y^n)=\gamma(x,\bar y)^n=\gamma(x^n,\bar y)=1$ for all $x \in \mathcal{Z}(G), y\in G, n\in \Z.$ This implies that $x^n \in \ker_L(\gamma) $, the left kernel of $\gamma.$ But since $\mathcal{Z}(E) = {MZ}(E)$, we have $\ker_L(\gamma)= \{x\in \mathcal{Z}(G)| [t(x),t(y)](t(y)\star t(x)) = 1$ for all $y\in G  \} = 1$ and so $x^n=1.$ Hence the result follows. 
\end{proof}

The following is a necessary condition for a multiplicative Lie algebra to be Lie
capable:
\begin{proposition}\label{Lie capable}
	If $G$ is Lie capable and $\frac{G}{G\star G}$ is of a finite exponent. Then $\mathcal{Z}(G)$ is bounded and exponent of $\mathcal{Z}(G)$ divides that of $\frac{G}{^M[G,G]}.$  
\end{proposition}
\begin{proof}
Consider the central extension
$$ 1 \longrightarrow \mathcal{Z}(E) \longrightarrow E \overset{\phi}\longrightarrow G\longrightarrow 1$$ of $G.$ Let $t$ be a section satisfying $(\phi \circ t)(x)=x$ for all $x\in G$, and $t(1)=1.$ If $s$ is another section , then $s(x)=z_x t(x)$ for some $z_x \in \mathcal{Z}(E). $ Thus $t(x)\star t(y)=s(x)\star s(y).$ The Lie commutators are independent of the chosen section.	 Thus we have a coordinatewise group homomorphism $\gamma:\mathcal{Z}(G)\times G \to\mathcal{Z}(E) $ given by $\gamma(x,y)=(t(x)\star t(y)).$ Fix $x \in \mathcal{Z}(G)$ then we have a group homomorphism $\gamma_x: G \to \mathcal{Z}(E) $ given by $\gamma_x(y)=t(x)\star t(y).$ The right kernel of $\gamma$ is $\ker_R(\gamma)=\cap_{x\in G}\ker(\gamma_x).$ Since $(G\star G)\subseteq \ker_R(\gamma) $ 
and exponent of $\frac{G}{G\star G}$ is finite, say $n$, we have $\gamma(x,\bar y^n)=\gamma(x,\bar y)^n=\gamma(x^n,\bar y)=1$ for all $x \in  \mathcal{Z}(G), y\in G, n\in \Z.$ This implies that $x^n \in \ker_L(\gamma) $, the left kernel of $\gamma.$ But $\ker_L(\gamma)=1$ and so $x^n=1.$ Hence the result follows. 
\end{proof}
\begin{remark}
	\begin{enumerate}
		\item 	If we replace ${G\star G}$ with $[G,G]$ in Proposition \ref{Lie capable}, the result remains valid. Furthermore, we can define $\gamma:\mathcal{Z}(G)\times G \to\mathcal{Z}(E)$ as $\gamma(x,y)=[t(x),t(y)]$.
		
		\item If there exists a multiplicative Lie algebra $E$ such that $G \cong \frac{E}{\mathcal{Z}(E)}$, and the condition $(E\star E) \subseteq Z(E)$ holds, then it follows that $[E,E] \subseteq LZ(E)$.
		\end{enumerate}
\end{remark}

 Now, we define a central ideal $\mathcal{Z}^*(G)$ in $G$ such that $\frac{G}{\mathcal{Z}^*(G)}$ is Lie capable. Furthermore, we establish that $G$ is Lie capable if and only if $\mathcal{Z}^*(G)$ equals the identity element.
 
\begin{definition}
A multiplicative Lie algebra $G$ is defined as Lie unicentral if its central extensions satisfy the condition $\mathcal{Z}^*(G) = \mathcal{Z}(G)$. Here $\mathcal{Z}^*(G)$ is obtained by taking the intersection of the images of the centers of all central extensions $E$ of $G$ through a homomorphism $\phi$ in the short exact sequence $1 \longrightarrow \ker(\phi) \longrightarrow E \overset{\phi}\longrightarrow G\longrightarrow 1$, i. e.
 $\mathcal{Z}^*(G)=\cap\{\phi(\mathcal{Z}(E))\ |\  1 \longrightarrow \ker(\phi) \longrightarrow E \overset{\phi}\longrightarrow G\longrightarrow 1$ is a central extension of $G \}.$ 
\end{definition}
In order to demonstrate that $\mathcal{Z}^*(G)$ is the smallest ideal of $G$ for which $\frac{G}{\mathcal{Z}^*(G)}$ is Lie capable,  we require the following proposition.
\begin{proposition}\label{C2}
	Let $\{N_i\}_{i\in I}$ be a family of ideals of a multiplicative Lie algebra $G.$
	If for each $i\in I,$ the factor multiplicative Lie algebra $\frac{G}{N_i}$  is Lie capable  , then so is $\frac{G}{\cap_{i\in I} N_i}$.
\end{proposition}

\begin{proof}
	For each $i\in I$, we have a central extension of $\frac{G}{ N_i}$
	$$ 1 \longrightarrow \mathcal{Z}(E_i) \longrightarrow E_i \overset{\phi_i}\longrightarrow \frac{G}{ N_i}\longrightarrow 1.$$
	
	From the pull back $G \times_{\frac{G}{N_i}} E_i = \{(g,e_i)\in G \times E_i \ | \ \phi_i(e_i)=gN_i \}$ which is a subalgebra of $G \times E_i.$ Set $E = \{(e_i)\in \prod_{i\in I}^{}E_i \ | \ \exists g\in G, \forall i\in I, \phi_i(e_i)=gN_i  \}$ which is a multiplicative Lie algebra with respect to coordinatewise operations. It is clear that $\prod_{i\in I}^{}\mathcal{Z}(E_i) \subseteq \mathcal{Z}(E).$ To show $\mathcal{Z}(E) \subseteq \prod_{i\in I}^{}\mathcal{Z}(E_i)$, assume that $(z_i)\in \mathcal{Z}(E) $ and $j\in I.$ Consider an arbitrary element $e_j \in E_j$, then for each $i\in I$ there exists  $e_i \in E_i$ such that $\phi_i(e_i)=\phi_j(e_j).$ This implies $(e_i)\in E$ and so $(e_i\star z_i)=(e_i)\star(z_i)= 1 $, and $(e_iz_i) = (z_ie_i).$ It follows that $e_j\in \mathcal{Z}(E_j)$, $(z_i)\in \prod_{i\in I}^{}\mathcal{Z}(E_i)$ and so  $\mathcal{Z}(E) = \prod_{i\in I}^{}\mathcal{Z}(E_i).$ Thus, we have a multiplicative Lie algebra isomorphism  $\frac{G}{\cap_{i\in I} N_i} \cong \frac{E}{\mathcal{Z}(E) }$ given by $gN \to e_{g,i}\mathcal{Z}(E)$ such that $\phi_i(e_{g,i}) = g\mathcal{Z}(E_i)$ for each $i \in I.$ Hence $\frac{G}{\cap_{i\in I} N_i}$ is Lie capable. 
	
\end{proof}

The following corollary shows that $\mathcal{Z}^*(G)$ is the smallest ideal of $G$ and $\frac{G}{\mathcal{Z}^*(G)}$ is Lie capable. 

\begin{corollary}\label{C3}
	 Let $G$ be a  multiplicative Lie algebra. Then $\frac{G}{\mathcal{Z}^*(G)}$ is Lie capable. If $\{N_i\}_{i\in I}$ be a family of all ideals of  $G$ such that $\frac{G}{N_i}$  is Lie capable. Then  $\mathcal{Z}^*(G) = \cap_{i\in I} N_i.$  
\end{corollary}

\begin{proof}
	Suppose $$ 1 \longrightarrow \ker(\phi) \longrightarrow E \overset{\phi}\longrightarrow G\longrightarrow 1 $$ is a central extension of $G$, then $\frac{G}{\phi(\mathcal{Z}(E))}\cong \frac{E}{\mathcal{Z}(E)}$ is Lie capable. Hence by  Proposition \ref{C2} $\frac{G}{\mathcal{Z}^*(G)}$ is Lie capable.
	For the central extension of $\frac{G}{ N_i}$
	$$ 1 \longrightarrow \mathcal{Z}(E_i) \longrightarrow E_i \overset{\phi_i}\longrightarrow \frac{G}{ N_i}\longrightarrow 1 $$
	form the pull back $G \times_{\frac{G}{N_i}} E_i = \{(g,e_i)\in G \times E_i \ | \ \phi_i(e_i)=gN_i \}.$ Then 
	$$ 1 \longrightarrow \ker(\chi_i) \longrightarrow G \times_{\frac{G}{N_i}} E_i \overset{\chi_i}\longrightarrow G\longrightarrow 1 $$
	is a central extension of $G$, where $\chi_i$ being the first projection. Let $(g,e_i) \in \mathcal{Z}(G \times_{\frac{G}{N_i}} E_i).$ Since $\chi_i((g,e_i))N_i=gN_i=\phi_i(e_i)=N_i. $ By definition $\mathcal{Z}^*(G) \subseteq \chi_i(\mathcal{Z}(G \times_{\frac{G}{N_i}} E_i))\subseteq N_i$ for all $i\in I.$ Hence $\mathcal{Z}^*(G) \subseteq \cap_{i\in I} N_i.$
	Specialize to $N_i=\mathcal{Z}^*(G)$ for some $i.$
\end{proof}

One also sees that there is a central extension of $G$ such that $\chi_i(\mathcal{Z}(G \times_{\frac{G}{N_i}} E_i))={Z}^*(G).$

\begin{corollary}\label{C4}
     $G$ is Lie capable  if and only if  $\mathcal{Z}^*(G)=1.$
\end{corollary}

\begin{proof}
	Since $$ 1 \longrightarrow \mathcal{Z}(E) \longrightarrow E \overset{\phi}\longrightarrow G\longrightarrow 1 $$ is a central extension of $G$ and $\mathcal{Z}^*(G) \subseteq \phi(\mathcal{Z}(E))=1.$ Hence the result follows.
\end{proof}
In order to provide an alternative characterization of $\mathcal{Z}^*(G)$ using a free presentation of $G$, we require the following proposition.
\begin{proposition}\label{C5}
	For every free presentation $$1\longrightarrow R \longrightarrow F \overset{\pi}\longrightarrow G\longrightarrow 1$$ and every
	central extension $$1\longrightarrow \ker(\phi) \longrightarrow K \overset{\phi}\longrightarrow G\longrightarrow 1$$ of a multiplicative Lie algebra $G$, we have $ \overline\pi (\mathcal{Z}(\frac{F}{^M[R,F]})) \subseteq \phi(\mathcal{Z}(K)). $
	
\end{proposition}

\begin{proof}
	By \cite{MS}, there exists a  multiplicative Lie algebra homomorphism $$\tilde\psi :\frac{F}{^M[R,F]}\to K $$ such that $\phi\circ \tilde{\psi}=\overline\pi$ with $\tilde\psi(\frac{R}{^M[R,F]})\subseteq \ker(\phi).$ We can easily verify that $$K = \ker(\phi)\im(\tilde\psi).$$
	 Since $$ \tilde\psi(\mathcal{Z}(\frac{F}{^M[R,F]}))\star  \tilde\psi (\frac{F}{^M[R,F]}) = \tilde\psi(\mathcal{Z}(\frac{F}{^M[R,F]})\star  (\frac{F}{^M[R,F]})) = 1 $$ and $$ [\tilde\psi(\mathcal{Z}(\frac{F}{^M[R,F]})),  \tilde\psi (\frac{F}{^M[R,F]})] =  \tilde\psi[(\mathcal{Z}(\frac{F}{^M[R,F]})), (\frac{F}{^M[R,F]})] = 1. $$ 
	 Also, $$ \tilde\psi (\mathcal{Z}(\frac{F}{^M[R,F]})) \star \ker(\phi) = 1 $$ and  $$ [\tilde\psi (\mathcal{Z}(\frac{F}{^M[R,F]})), \ker(\phi)] = 1. $$   This implies  $ \tilde\psi (\mathcal{Z}(\frac{F}{^M[R,F]})) \subseteq \mathcal{Z}(K). $ Hence, we have  $\phi (\tilde\psi (\mathcal{Z}(\frac{F}{^M[R,F]})) \subseteq \phi(\mathcal{Z}(K)) $ which completes the proof.
\end{proof}

\begin{corollary}\label{C6}
		For any free presentation $$1\longrightarrow R \longrightarrow F \overset{\pi}\longrightarrow G\longrightarrow 1$$ of $G$, 
		$\mathcal{Z}^*(G) = \overline\pi \bigg(\mathcal{Z}\bigg(\frac{F}{^M[R,F]}\bigg)\bigg). $
\end{corollary}
\begin{proof}
	It is clear by Proposition \ref{C5} that $ \overline\pi (\mathcal{Z}(\frac{F}{^M[R,F]})) \subseteq \mathcal{Z}^*(G).$ By definition of $\mathcal{Z}^*(G)$, we have  $\mathcal{Z}^*(G) \subseteq \overline\pi (\mathcal{Z}(\frac{F}{^M[R,F]})).$
\end{proof}
\begin{definition}\label{C7}
	A central extension $$1\longrightarrow N \longrightarrow G^* \overset{\phi}\longrightarrow G\longrightarrow 1$$ of multiplicative Lie algebras is called a stem extension if $N \subseteq {^M[G^*,G^*]}.$ If in addition $N \cong \tilde M(G)$ then the above extension is called a stem cover. In this case, $G^*$ is said to be a cover (or covering algebra) of $G.$
\end{definition}
The following corollaries are immediate consequences of Proposition \ref{P1} and \ref{P2}.
\begin{corollary}\label{P3}
	Let $$1\longrightarrow R \longrightarrow F \overset{\pi}\longrightarrow G\longrightarrow 1$$  be a free presentation of a perfect multiplicative Lie algebra $G.$ Then $\frac{^M[F,F]}{^M[R,F]}$ is a cover of $G.$
\end{corollary}

\begin{corollary}\label{P4}
	Suppose  $$1\longrightarrow N \longrightarrow G^* \overset{\phi}\longrightarrow G\longrightarrow 1$$ is a universal central extension of a perfect multiplicative Lie algebra $G$, then $N \cong \tilde M(G) $ and $G^*$ is a cover of $G.$
\end{corollary}
\begin{lemma}\label{C8}
	Let $A_1 , A_2 $ and $A_1 \times A_2$ be abelian groups with trivial multiplicative Lie algebras and $G$ be any  multiplicative Lie algebra. Then $H^2_{ML}(G,A_1 \times A_2) \cong H^2_{ML}(G,A_1) \times H^2_{ML}(G,A_2).  $
\end{lemma}

\begin{proof}
	The  map $\psi:Z^2_{ML}(G,A_1 \times A_2) \to Z^2_{ML}(G,A_1) \times Z^2_{ML}(G,A_2)$ by $\psi((f,h))=((\pi_1\circ f,\pi_1\circ h),(\pi_2\circ f,\pi_2\circ h))$, where $\pi_i:A_1\times A_2\to A_i$, $i=1,2$ gives the required isomorphism.
\end{proof}

The following theorem ensures the existence of a cover of finite multiplicative Lie algebra.

\begin{theorem}\label{C9}
	Any finite multiplicative Lie algebra $G$ has at least one cover $G^*.$
\end{theorem}

\begin{proof}
	We write $\tilde M(G) = {<(f_1,h_1)>\times <(f_2,h_2)> \times \cdots \times <(f_m,h_m)>} $ for suitable ${<(f_1,h_1)> ,\cdots , <(f_m,h_m)>} \in \tilde M(G).$ If $d_i$ denote the order of $<(f_i,h_i)>$ then $<(f_i,h_i)>$ contains a multiplicative Lie 2-cocycle $(\alpha_i,\beta_i)$ of order $d_i.$ Let $A_i$ be the group of all $d_i$th roots of $1$ in $\C^*$ and let $A = A_1 \times A_2 \times \cdots \times A_m.$ Then the choice of $d_i$ guarantees that $A\cong \tilde M(G). $ Let $\bar{\psi}: H^2_{ML}(G,A_i) \to \tilde M(G) $ be the homomorphism induced by the natural injection $\psi:A_i \to \C^*.$ Since $(\alpha_i,\beta_i)$ is an $A_i$-valued multiplicative Lie 2-cocycle, there exists $(\gamma_i,\xi_i)\in H^2_{ML}(G,A_i) $ such that $\bar{\psi}((\gamma_i,\xi_i))=(f_i,h_i).$ Let $(\gamma,\xi)$ be the image of $((\gamma_1,\xi_1), \cdots , (\gamma_m,\xi_m))$ under the isomorphism $\prod_{i=1}^{m}H^2_{ML}(G,A_i)\to H^2_{ML}(G,A)$ in Lemma \ref{C8}. Suppose  $$1\longrightarrow A \longrightarrow G^* \longrightarrow G\longrightarrow 1$$ is a central extension associated with $(\gamma,\xi).$ To prove our theorem it suffices to verify that $A\subseteq {^M[G^*,G^*]}.$  Consider the homomorphism   $\delta:Hom(A,\C^*)\to \tilde M(G) $ associated with $(\gamma,\xi)$ (See \cite{RLS}). Then $\delta$ is surjective and $\tilde M(G)\cong\im(\delta)\cong {^M[G^*,G^*]}\cap A $ by \cite{RLS} . But $A\cong \tilde M(G)$, so $A\subseteq {^M[G^*,G^*]}.$
\end{proof}

\begin{remark}\label{C10}
Cover of a  multiplicative Lie algebra $G$ need not be unique.
\end{remark}

\begin{example}
	Consider the Klein four-group $V_4$ with trivial multiplicative Lie algebra structure. We know  that $\tilde M(V_4) = V_4 $ from \cite{ADSS}. Let $(\gamma,\xi)$ be a multiplicative Lie 2-cocycle of a trivial multiplicative Lie algebra $V_4$  with coefficient in
	an abelian group $V_4$ with trivial Lie product. By construction $$1\longrightarrow \tilde M(V_4) \longrightarrow G^* \longrightarrow V_4 \longrightarrow 1$$ is a central extension associated with $(\gamma,\xi).$  Then by \cite{RLS}, we have $G^* = \tilde M(V_4) \times V_4 = V_4 \times V_4 $ is a multiplicative Lie algebra with respect to the  operations  $$(a,x)\cdot(b,y) = (ab\gamma(x,y),xy)$$ and $$(a,x)\star(b,y) = (\xi(x,y),x\star y) = (\xi(x,y), 1). $$ Take $\xi(x,y)=1$ for all $x,y \in V_4.$ Also, ${^M[G^*,G^*]} = V_4$ because $(a,x)\cdot(b,y)\cdot(a^{-1}\gamma(x,x^{-1})^{-1},x^{-1})\cdot(b^{-1}\gamma(y,y^{-1})^{-1},y^{-1}) = \gamma(x,x)\gamma(y,x)  $ and $\gamma(x,x) \neq \gamma(y,x).$ Hence $G^*$ is a cover   of $V_4$ with trivial multiplicative Lie algebra structure which depends on the choice of group 2-cocycle $\gamma:V_4 \times V_4 \to V_4.$
	
	Moreover, if we take $\xi(x,y) = x$, then $G^*$ is a cover   of $V_4$ with non-trivial multiplicative Lie algebra structure which also depends on the choice of group 2-cocycle $\gamma:V_4 \times V_4 \to V_4.$
\end{example}

\begin{remark}
	If $f: G \to H$ is an epimorphism of multiplicative Lie algebras, then $f(\mathcal{Z}(G))\subseteq \mathcal{Z}(H).$
\end{remark}

The subsequent theorem provides an alternative representation of $\mathcal{Z}^*(G)$ using covers. In particular, it states the existence of a central extension of $G$ where the image of its center corresponds to $\mathcal{Z}^*(G)$.
\begin{theorem}\label{C12}
	Let $G$ be a multiplicative Lie algebra such that its Schur multiplier is finite
 and let $$1\longrightarrow N \longrightarrow G^* \overset{\phi}\longrightarrow G\longrightarrow 1$$ be a stem cover of $G.$ Then $\mathcal{Z}^*(G) = \phi(\mathcal{Z}(G^*)).$
\end{theorem}

\begin{proof}
Let $$1\longrightarrow R \longrightarrow F \overset{\pi}\longrightarrow G\longrightarrow 1$$ be a free presentation of multiplicative Lie algebra $G.$ By \cite{MS}, there exists an ideal $T$ of $F$
such that $G^*\cong \frac{F}{T} ,  N\cong \frac{R}{T}$, and $\frac{R}{^M[R,F]}\cong \tilde M(G)\frac{T}{^M[R,F]}. $ So we may regard $\phi$ as an epimorphism of multiplicative Lie algebra $\frac{F}{T}$ onto $G$  such that $\ker \phi \cong \frac{R}{T}.$ Let $\nu : F \to \frac{F}{T} $ be natural epimorphism of multiplicative Lie algebras.  Suppose $S = \nu^{-1}(\mathcal{Z}(\frac{F}{T}))$ , then $S$ is an ideal of $F$ such that  ${^M[R,F]} \subseteq S$. So $\mathcal{Z}(\frac{F}{^M[R,F]})\subseteq \frac{S}{^M[R,F]}.$ Also, $^M[\frac{S}{^M[R,F]},\frac{T}{^M[R,F]}] \subseteq \frac{T}{^M[R,F]}\cap \frac{{^M[F,F]}}{^M[R,F]} = 1$ and  $$1\longrightarrow  \frac{T}{^M[R,F]} \longrightarrow  \frac{F}{^M[R,F]} \longrightarrow \frac{F}{T}\longrightarrow 1$$ is exact. So $\frac{S}{^M[R,F]} = \mathcal{Z}(\frac{F}{^M[R,F]}).$ Thus we have
$\mathcal{Z}^*(G) = \overline\pi (\mathcal{Z}(\frac{F}{^M[R,F]}))  = \overline\pi (\frac{S}{^M[R,F]}) = \pi(S) = \phi(\mathcal{Z}(\frac{F}{T})) = \phi(\mathcal{Z}(G^*)). $
\end{proof}

\begin{theorem}\label{C13}
		Let $G$ be a multiplicative Lie algebra such that $\mathcal{Z}(G) \subseteq \mathcal{Z}^*(G)$ then for every stem extension $$1\longrightarrow N \longrightarrow K \overset{\phi}\longrightarrow G\longrightarrow 1$$ we have  $\phi(\mathcal{Z}(K))=\mathcal{Z}(G).$
\end{theorem}

\begin{proof}
By definition of $\mathcal{Z}^*(G)$, we have  $\mathcal{Z}^*(G) \subseteq \phi(\mathcal{Z}(K))\subseteq \mathcal{Z}(G) $ for every stem extension $$1\longrightarrow H \longrightarrow K \overset{\phi}\longrightarrow G\longrightarrow 1.$$ Hence the proof follows. 	
\end{proof}

\begin{proposition}\label{C14}
	Let $G$ be a multiplicative Lie algebra with a free presentation $$1\longrightarrow R \longrightarrow F \overset{\pi}\longrightarrow G\longrightarrow 1.$$
	If $S$ is an ideal in $F$ with $\frac{S}{R} = N$, then the following sequence is exact:
	
	$$1\longrightarrow \frac{^M[S,F]\cap R}{^M[R,F]} \longrightarrow \tilde M(G) \overset{\sigma}\longrightarrow \tilde M(\frac{G}{N})\longrightarrow \frac{^M[G,G]\cap N}{^M[N,G]}\longrightarrow 1.$$
\end{proposition}

\begin{proof}
	By \cite{RLS}, we have $\tilde M(G) = \frac{^M[F,F]\cap R}{^M[R,F]} $ and $\tilde M(\frac{G}{N}) = \frac{^M[F,F]\cap S}{^M[S,F]}.$ Since $\frac{^M[G,G]\cap N}{^M[N,G]} =  \frac{(^M[F,F]\cdot R)\cap S}{^M[S,F]\cdot R} = \frac{(^M[F,F]\cap S)\cdot R}{^M[S,F]\cdot R}$, therefore one can easily check that the following sequence is exact:
	$$1\longrightarrow \frac{^M[S,F]\cap R}{^M[R,F]} \longrightarrow  \frac{^M[F,F]\cap R}{^M[R,F]}  \overset{\sigma}\longrightarrow \frac{^M[F,F]\cap S}{^M[S,F]}\longrightarrow \frac{(^M[F,F]\cap S)\cdot R}{^M[S,F]\cdot R}\longrightarrow 1,$$ which gives the result.
	
\end{proof}

\begin{corollary}\label{C15}
	Let $N$ be an ideal in a multiplicative Lie algebra $G$ and $\tilde M(G) = 1.$ Then $\tilde M(\frac{G}{N})\cong  \frac{^M[G,G]\cap N}{^M[N,G]}.$
\end{corollary}

\begin{corollary}\label{C16}
	Let $G$ be a $M\mathcal{Z}$-nilpotent multiplicative Lie algebra and $f$ be a homomorphism from $G$ onto another multiplicative Lie algebra $T$. If $ \ker f \subseteq {^M[G,G]}$ and $\tilde M(T)$ is trivial, then $f$ is an isomorphism.
	In particular, if $\tilde M(\frac{G}{^M[G,G]}) = 1$, then so is $\tilde M(G).$
\end{corollary}

\begin{proof}
	Put $K = \ker f $, then  $\tilde M(\frac{G}{K}) = 1.$ Hence Proposition \ref{C14} implies that $\frac{^M[G,G]\cap K}{^M[K,G]} = 1$ and so $K =  {^M[K,G]}. $ Now assume that $K_0 = K$ and $K_{n+1} = {^M[K_n,G]}$ for all $n \geq 0.$ Then $K = K_n\subseteq M_n(G).$ Since $G$ is $M\mathcal{Z}$-nilpotent, it follows that $K = 1.$ Hence $f$ is an isomorphism.  
\end{proof}

One observes  that the above result gives a criterion for a multiplicative Lie algebra to have a trivial Schur multiplier.

\begin{theorem}\label{C17}
	Let $N$ be a central ideal in a multiplicative Lie algebra $G.$ Then $N \subseteq \mathcal{Z}^*(G) $ if and only if $\sigma:\tilde M(G)\to  \tilde M(\frac{G}{N}) $ is a monomorphism.
\end{theorem}

\begin{proof}
	Assume that $G = \frac{F}{R}$ and $N = \frac{S}{R}$ as in the Proposition \ref{C14}. By construction, the kernel of the map $\sigma:\tilde M(G)\to  \tilde M(\frac{G}{N}) $ is $\frac{^M[S,F] \cap R}{^M[R,F]}.$ Thus we only need to verify that ${^M[S,F] \cap R} = {^M[R,F]} $ if and only if  $N \subseteq \mathcal{Z}^*(G) .$ By Corollary \ref{C6} , we have $\mathcal{Z}^*(G) =  \overline\pi (\mathcal{Z}(\frac{F}{^M[R,F]})).$ Thus, we have $\overline\pi (\frac{S}{^M[R,F]}) \subseteq \mathcal{Z}^*(G)  $ if and only if $\frac{S}{^M[R,F]}\subseteq \mathcal{Z}(\frac{F}{^M[R,F]}).$ Since $\overline\pi (\frac{S}{^M[R,F]}) = N $, the result follows.
	
\end{proof}


\begin{lemma}\label{P5}
	Assume  $G$ is a finite perfect multiplicative Lie algebra. If  $$1\longrightarrow 1 \longrightarrow H \overset{\phi}\longrightarrow G\longrightarrow 1$$ is a  universal extension then so is  $$1\longrightarrow 1 \longrightarrow G \overset{id}\longrightarrow G\longrightarrow 1.$$ 
\end{lemma}

\begin{proof}
	Let  $$1\longrightarrow N \longrightarrow K \overset{\psi}\longrightarrow G\longrightarrow 1$$ be any central extension of $G.$ Then there is a unique homomorphism $\theta: H \to K $ such that the corresponding diagram is commutative. Hence $\phi = \psi \circ \theta.$ Since $\phi$ is an isomorphism. Let $\beta = \theta \circ \phi^{-1}.$ Then $\beta: G\to K$ is a homomorphism such that $\psi \circ \beta = \psi \circ\theta \circ \phi^{-1} = \phi \circ \phi^{-1}= 1 $ on $G.$ Thus the extension  $$1\longrightarrow N \longrightarrow K \overset{\psi}\longrightarrow G\longrightarrow 1$$ splits. Since $G$ is perfect and every central extension of $G$ splits, by \cite{RLS}, $$1\longrightarrow 1 \longrightarrow G \overset{id}\longrightarrow G\longrightarrow 1$$ is universal.  
\end{proof}

\begin{theorem}\label{P6}
	Let $G$ be a finite perfect  multiplicative Lie algebra with $\tilde M(G)=1.$ 
	 
	$(i)$ $H^2_{ML}(G,A)=1$,  $A$ is an abelian group with  trivial multiplicative Lie algebra.
	
	$(ii)$ If $N$ is a central ideal of $G$, then $N\cong \tilde M(\frac{G}{N})$ and $G$ is a cover of $\frac{G}{N}.$
\end{theorem}

\begin{proof}
$(i)$	By \cite{RLS},  $\tilde M(G)=1$ implies that  $$1\longrightarrow 1 \longrightarrow \frac{^M[F,F]}{^M[R,F]} \overset{\phi}\longrightarrow G\longrightarrow 1$$ is a universal central extension. Then by Lemma \ref{P5} ,  $$1\longrightarrow 1 \longrightarrow G \overset{id}\longrightarrow G\longrightarrow 1$$ is  universal. By \cite{RLS}, every central extension of $G$ splits. Hence $H^2_{ML}(G,A)=1.$

$(ii)$ Since  $$1\longrightarrow 1 \longrightarrow G \overset{id}\longrightarrow G\longrightarrow 1$$ is  universal, we conclude that the extension  $$1\longrightarrow N \longrightarrow G \longrightarrow \frac{G}{N}\longrightarrow 1$$ is  universal. By Corollary \ref{P4}, we have $N\cong \tilde M(\frac{G}{N})$ and $G$ is a cover of $\frac{G}{N}.$ 
\end{proof}

\noindent{\bf Acknowledgement:}
The first named author sincerely thanks IIIT Allahabad and the University grant commission (UGC), Govt. of India, New Delhi for the research fellowship.

\end{document}